\documentclass[a4paper, 12 pt]{article}
\usepackage{amscd}
\usepackage{amssymb}
\usepackage{amsfonts}
\usepackage{amsmath}
\allowdisplaybreaks
\usepackage{amsthm}
\usepackage{bbm}
\usepackage{CJK}
\usepackage{fancyhdr}
\usepackage{graphicx}

\usepackage{geometry}
\usepackage{ifpdf}
\ifpdf
  \usepackage[colorlinks=true,linkcolor=blue,final,backref=page,hyperindex]{hyperref}
\else
  \usepackage[colorlinks,final,backref=page,hyperindex,hypertex]{hyperref}
\fi
\usepackage{tikz}
\usepackage[active]{srcltx}

\usepackage{indentfirst}
\usepackage{latexsym}
\usepackage{mathrsfs}
\usepackage{xypic}

\usepackage{pxfonts}

\renewcommand{\theequation}{\arabic{section}.\arabic{equation}}
\def\vbar{\mathchoice{\vrule height6.3ptdepth-.5ptwidth.8pt\kern- .8pt}
{\vrule height6.3ptdepth-.5ptwidth.8pt\kern-.8pt} {\vrule
height4.1ptdepth-.35ptwidth.6pt\kern-.6pt} {\vrule
height3.1ptdepth-.25ptwidth.5pt\kern-.5pt}}

\newcommand{\QC}{{\rm QC}}
\newcommand{\GDer}{{\rm GDer}}
\newcommand{\ZDer}{{\rm ZDer}}

\newcommand{\QDer}{{\rm QDer}}
\newcommand{\Der}{{\rm Der}}

\newcommand{\R}{\mathcal{R}}

\newcommand{\g}{\mathfrak{g}}

\def\<{\langle}
\def\>{\rangle}
\def\a{\alpha}
\def\b{\beta}
\def\f{\phi}
\def\p{\psi}

\def\c{\cdot}

\def\g{\gamma}

\def\l{\lambda}

\def\o{\otimes}

\def\pr{\partial}
\def\R{\mathcal{R}}

\theoremstyle{definition}
\newtheorem{df}{Definition}[section]
\theoremstyle{plain}
\newtheorem{thm}{Theorem}[section]
\newtheorem{cor}{Corollary}[section]
\newtheorem{rem}{Remark}[section]

\newtheorem{prop}{Proposition}[section]
\newtheorem{lem}{Lemma}[section]

\theoremstyle{definition}

  \theoremstyle{definition}
    \newtheorem{exa}[thm]{Example}

\newcommand{\bmx}{\begin{pmatrix}}
\newcommand{\emx}{\end{pmatrix}}

\date{}
\begin{document}
\title{ \bf Cohomology and conformal derivations of BiHom-Lie conformal superalgebras }
\author{\bf Taoufik Chtioui}
\author {Taoufik Chtioui 
 \footnote { Corresponding author,  E-mail: chtioui.taoufik@yahoo.fr} \\
{\small  University of Sfax, Faculty of Sciences Sfax,  BP
1171, 3038 Sfax, Tunisia}}
 \maketitle
\begin{abstract}
 In this paper,  we introduce  the notion of BiHom-Lie conformal superalgebras. We develop its  representation theory and define the cohomology group with coefficients in a module. Finally, we introduce conformal derivations of BiHom-Lie conformal superalgebras and study some of their properties.
\end{abstract}

\tableofcontents
\numberwithin{equation}{section}
\section*{Introduction}
\def\theequation{0. \arabic{equation}}
\setcounter{equation} {0}

The notion of Lie conformal (super)algebras are introduced by Kac in \cite{Kac98} in which he gave  an axiomatic description of the singular part of the operator product expansion of chiral fields in conformal field theory. On the other hand,  it is a useful tool to study vertex (super)algebras and has many applications in the theory of Lie superalgebras. Moreover, 
 It is equivalent to the notion of a linear Hamiltonian operator introduced in  \cite{gelfand}, 
In \cite{Zhao2017}, Zhao et al.  developed deformation of Lie conformal superalgebras and   introduced derivations of  Lie conformal superalgebras and study their properties.

In \cite{Ammar2010}, authors  introduced the notion of Hom-Lie superalgebras and 
they gave a classification of Hom-Lie admissible superalgebras.  
Later,  Makhlouf et al. studied the representation and the cohomology of Hom-Lie superalgebras in  \cite{Ammar2013} and calculated the derivations and the second cohomology group of $q$-deformed Witt superalgebra. 
In \cite{Yuan17}, Yuan introduced  the notion of Hom-Lie conformal superalgebra and proved that a Hom-Lie conformal superalgebra is equivalent to a Hom-Gel'fand-Dorfman superbialgebra.
In \cite{Zhao2016},  the authors developed cohomology theory of Hom-Lie conformal algebras, discussed  some applications to the study of deformations of regular Hom-Lie conformal algebras and  introduced  derivations of multiplicative Hom-Lie conformal
algebras.
Motivated by these results, authors, in \cite{hom lie conf super},  introduced  the notion of representation theory of Hom-Lie conformal superalgebras and  discuss the cases of adjoint representations. Furthermore,  they developed  their  cohomology group and discuss
some applications to the study of deformation theory. 

Motivated by a categorical study
of Hom-algebra and new type of categories, the authors  introduced, in \cite{bihom-lie},
a generalized algebraic structure dealing with two commuting multiplicative linear maps, called BiHom-algebras including BiHom-associative algebras and BiHom-Lie algebras.
Recently,  Zhao, Yuan and Chen  developed the cohomology  and deformation theory 
of BiHom-Lie conformal algebras and investigated  the notion of  conformal derivations of  BiHom-Lie conformal algebras  in \cite{bihom lie conf alg}.

The main goal of the present work  is to introduce the notion of BiHom-Lie conformal superalgebras, study their cohomology theory and investigate the notion of generalized conformal derivations. 

The paper is organized as follows.
In section 1, we recall some basic definitions and results about the notion of Lie conformal (super)algebras. 
In Section 2,  we introduce  the notion  BiHom-Lie conformal superalgebras illustrated by some examples  give some related results. 
In Section 3, we develop  representation and cohomology theory of BiHom-Lie conformal superalgebras.  
Section 4 is devoted to the study  of derivations and  generalized derivations of BiHom-Lie conformal superalgebras and  their properties.

\section{Preliminaries}
\def\theequation{\arabic{section}.\arabic{equation}}
\setcounter{equation} {0}

Throughout the paper, all algebraic systems are supposed to be over a field $\mathbb{C}$,
of characteristic $0$ and denote by $\mathbb{Z}_{+}$ the set of all nonnegative integers and by $\mathbb{Z}$ the set of all integers.

Let $V$ be a superspace that is a $\mathbb{Z}_2$-graded linear space with a direct sum $V = V_0 \oplus V_1$. The elements
of $V_j , j =\{0,1\}$, are said to be homogenous and of parity $j$. The parity of a homogeneous element $x$
is denoted by $|x|$. Throughout what follows, if $|x|$ occurs in an expression, then it is assumed that $x$ is homogeneous and that the expression extends to the other elements by linearity.

\begin{df}
A Lie conformal superalgebra $\R$ is a left $\mathbb{Z}_{2}$-graded $\mathbb{C}[\partial]$-module, and for any $n\in\mathbb{Z}_{\geq0}$ there is a
family of $\mathbb{C}$-linear  $n$-products from $\R\otimes\R$ to $\R$ satisfying the following conditions
\begin{enumerate}
\item[$(\rm C0)$]  For any $a,b\in\R$, there is an $N$ such that $a_{(n)}b=0$ for $n\gg N$,
\item[$(\rm C1)$]  For any $a,b\in\R$ and $n\in\mathbb{Z}_{\geq0}$, $(\partial a)_{(n)}b=-n a_{(n)}b$,
\item[$(\rm C2)$]  For any $a,b\in\R$ and $n\in\mathbb{Z}_{\geq0}$, $$a_{(n)}b=-(-1)^{|a||b|}\sum_{j=0}^{\infty}(-1)^{j+n}\frac{1}{j!}\partial^j(b_{(n+j)}a),$$
\item[$(\rm C3)$]  For any $a,b,c\in\R$ and $m,n\in\mathbb{Z}_{\geq0}$,
$$a_{(m)}(b_{(n)}c)=\sum_{j=0}^{m}(_{j}^{m})(a_{(j)}b)_{(m+n-j)}c+(-1)^{|a||b|}b_{(n)}(a_{(m)}c).$$
\end{enumerate}
(Convention: $a_{(n)}b = 0$ if  $n < 0$). Note that if we define $\lambda$-bracket $[-_{\lambda}-]$:
\begin{eqnarray}
[a_\lambda b]=\sum_{n=0}^{\infty}\frac{\lambda^n}{n!}a_{(n)}b,\, a,b\in\R.
\end{eqnarray}
That is, $\R$ is a Lie conformal superalgebra if and only if $[-_{\lambda}-]$ satisfies the following axioms
\begin{eqnarray*}
&&(\rm C1)_{\lambda}\ \ \ \ {\rm  Conformal\ sesquilinearity}:\ \, [(\partial a)_\lambda b]=-\lambda[a_\lambda b];\\
&&(\rm C2)_{\lambda}\ \ \ \ {\rm Skew-symmetry}:     \ \ \ \ \ \ \ \ \ \ \ \, [a_\lambda b]=-(-1)^{|a||b|}[b_{-\partial-\lambda}a];\\
&&(\rm C3)_{\lambda}\ \ \ \ {\rm Jacobi\ identity}:  \ \ \ \ \ \ \ \ \ \ \ \ \ \ \  \, [a_\lambda[b_\mu c]]=[[a_\lambda b]_{\lambda+\mu}c]+(-1)^{|a||b|}[b_\mu[a_\lambda c]].
\end{eqnarray*}
\end{df}

Similarly, an associative conformal  superalgebra is a  left $\mathbb{Z}_{2}$-graded $\mathbb{C}[\partial]$-module$\R$  equipped with a $\mathbb{C}$-linear map $\c_\l : \R\otimes \R \to \R[\l]$ such that $\R_{i\l} R_j \subseteq \R_{i+j}[\l]$ and
\begin{eqnarray*}
&&(\rm C1)_{\lambda}\ \ \ \ {\rm  Conformal\ sesquilinearity}:\ \, [(\partial a)_\lambda b]=-\lambda[a_\lambda b];\\
&&(\rm C2)'_{\lambda}\ \ \ \ {\rm associativity}:     \ \ \ \ \ \ \ \ \ \ \ \,
(a_\l b)_{\l+\mu} c=a_\l (b_\mu c).
\end{eqnarray*}
$\R$ is called finite if it is finitely generated as a $\mathbb{C}[\pr]$-module.
\begin{prop}
Let $(\R,\c_\l)$ be an associative conformal superalgebra. Then $(\R,[\c_\l \c])$ is a Lie conformal superalgebra, where for any homogenous elements $a,b \in \R$
$$[a_\l b]= a_\l b- (-1)^{|a||b|} b_{-\l -\pr} a.$$
\end{prop}

\begin{exa}
Let $\R$ be a Lie conformal superalgebra and
let $B$ be a commutative associative (ordinary) superalgebra.
Then $\R\otimes B$ carries a Lie conformal superalgebra structure defined
as follows. The $\mathbb{C}[\pr]$-module structure is given by
$\pr(r\otimes b)= (\pr r)\otimes b$ ($r\in \R,\; b\in B$), and
the $\l$-bracket by
\begin{equation*}
[(r\otimes b)_{\l}(r' \otimes b')]=(-1)^{|b||r'|}
[r_{\l}r']\otimes(bb').
\end{equation*}
Notice that if $\R$ is finite and $B$ is finite-dimensional, then $\R
\otimes B$
is also finite.
\end{exa}

\begin{df}
A Hom-Lie conformal superalgebra $R=R_{\bar{0}}\oplus R_{\bar{1}}$ is a  $\mathbb{Z}_2$-graded $\mathbb{C}[\partial]$-module equipped with an even linear endomorphism
$\a$ such that $\a\partial=\partial\a$, and  a $\mathbb{C}$-linear map
\begin{eqnarray*}
R\o R\rightarrow \mathbb{C}[\lambda]\o R,  ~~~a\o b\mapsto [a_\lambda b]
\end{eqnarray*}
such that $[R_{i \lambda}R_{j}]\subseteq R_{i+ji}[\lambda]$, $i,j \in \mathbb{Z}_2$, and the following axioms hold for $a,b,c\in R$
\begin{eqnarray}
&&[\partial a_{\lambda}b] =-\lambda[a_\lambda b],[ a_{\lambda}\partial b] =(\partial+\lambda)[a_\lambda b],\\
&&[a_{\lambda}b]=-(-1)^{|a||b|}[b_{-\lambda-\partial}a],,\\
&&[\a(a)_\lambda[b_\mu c]]=[[a_{\lambda}b]_{\lambda+\mu}\a(c)]+(-1)^{|a||b|}[\a(b)_{\mu}[a_{\lambda}c]].
\end{eqnarray}
\end{df}

\begin{df}
 A BiHom-Lie conformal algebra is a  BiHom-conformal algebra $(A,[\c_{\lambda}\c],\a,\b)$
 such that the following axioms hold for any $a,b,c\in A$
\begin{align}
&\a([a_\lambda b]) = [\a(a)_\lambda \a(b)], \b([a_\lambda b]) = [\b(a)_\lambda \b(b)]\\
&[\b(a)_{\lambda}\a(b)]=-[\b(b)_{-\lambda-\partial}\a(a)],\\
&[\a\b(a)_\lambda[b_\mu c]]=[[\b(a)_{\lambda}b]_{\lambda+\mu} \b(c)]+[\b(b)_{\mu}[\a(a)_{\lambda}c]].
\end{align}
In particular, if $\a,\b$ are algebra
isomorphisms, then $(A,[\c_{\lambda}\c],\a,\b)$ is called regular.
\end{df}

\begin{df}\label{modulebihomlie conformal}
A module $(V, \f, \p)$ over a BiHom-Lie conformal algebra $(A[\c_\l], \a,\b)$ is a $\mathbb{C}[\partial]$-
module endowed with two commuting $\mathbb{C}$-linear maps $\f, \p$ and a $\mathbb{C}$-bilinear map $A\o V\rightarrow V[\lambda], a\o v\mapsto a_\lambda v$,
such that for $a,b\in A, v\in V$:
 \begin{align}
&\p \circ \partial =\partial \circ \p, \f \circ \partial =\partial \circ \f,\\
&(\partial a)_{\lambda}v=-\lambda(a_\lambda v),a_\lambda(\partial v)=(\partial+\lambda)a_{\lambda} v,\\
&\f(a_\lambda v)=\a(a)_\lambda(\f(v)),\p(a_\lambda v)=\b(a)_\lambda(\p(v)),\\
&[\b(a)_\lambda b]_{\lambda+\mu}\p(v)=\a\b(a)_\lambda(b_\mu v)-\b(b)_{\mu}(\a(a)_\lambda v). \label{repbihomlie conf}
 \end{align}
\end{df}

\begin{rem}
By  Definition \ref{modulebihomlie conformal},  it is easy to see hat a module over a BiHom-Lie conformal algebra $A$ in
a finite $\mathbb{C}[\pr]$-module $V$ is the same as a homomorphism of BiHom-Lie conformal algebra $\rho: A \to gc(V)$ which is called a representation of $A$.  The identity \eqref{repbihomlie conf} can be written as
\begin{align*}
\rho([\b(a)_\lambda b])_{\lambda+\mu}\circ \p=\rho(\a\b(a))_\l \circ \rho(b)_\mu-\rho(\b(b))_\mu \circ \rho(\a(a))_\l,
\end{align*}
for all $a,b \in A$.  In addition, $\rho$ satisfy $\rho(\pr a)_\l v=-\l \rho(a)_\l v,\ \forall a \in A, v \in V$.\\
The adjoint representation of $A$ is denoted by $ad$, i.e. $ad(a)_\l(b)=[a_\l b]$, where $a,b \in A$.
\end{rem}

\begin{df}
A BiHom-associative conformal superalgebra $\R$ is a  $\mathbb{Z}_2$-graded $\mathbb{C}[\partial]$-module equipped with two even linear endomorphisms
$\a, \b$ and endowed with a $\lambda$-product from $\R\o \R$ to $\mathbb{C}[\partial]\o \R$, for any $a,b,c\in \R$, satisfying the following conditions:
\begin{eqnarray}
&&(\partial a)_\lambda b=-\lambda a_{\lambda}b, a_{\lambda}(\partial b)=(\partial+\lambda)a_{\lambda}b,\\
&& \a\partial=\partial\a,  \b \partial=\partial\b,   \\
&& \a(a_\l b)=\a(a)_\l \a(b),  \b(a_\l b)=\b(a)_\l \b(b),  \\
&& \a(a)_{\lambda}(b_{\mu}c)=(a_\lambda b)_{\lambda+\mu}\a(c).
\end{eqnarray}
\end{df}

\begin{df}
A subset $U$ of a BiHom-conformal (super)-algebra $(A,\c_\l,\a,\b)$ is called a BiHom-conformal subalgebra if $\a(U) \subseteq A,\ \b(U)\subseteq A$ and $a_\l b \in U$, for any $a,b \in U$.

A subset $U$ $A$ is called a a left BiHom-conformal ideal (resp. right BiHom-conformal ideal) if $\a(U) \subseteq A,\ \b(U)\subseteq A$ and $a_\l b \in U$, for any $a \in A, b \in U$.
(resp. for any $a \in U, b \in A$).
\end{df}

\section{Definitions and Examples}
In this section, we introduce the notion of BiHom-Lie conformal superalgebra, which is a BiHom-generalization
of (Hom)-Lie conformal superalgebras and also a superanalogue of (Bi)Hom-Lie conformal
algebras.

\begin{df}
A BiHom-Lie conformal superalgebra  $\R=\R_0\oplus \R_1$
 is a $\mathbb{Z}_2$-graded $\mathbb{C}[\partial]$-module equipped with two commuting  linear maps
$\a,\b$ and a $\lambda$-bracket $[\c_{\lambda}\c]$ which defines a $\mathbb{C}$-linear map from $\R \otimes \R$ to  $\R[\lambda]=\mathbb{C}[\lambda]\otimes \R$ such that
 $[\R_{i\l }\R_j]\subseteq \R_{i+j}[\l],\ i,j \in \mathbb{Z}_2$ and
 the following axioms hold for all homogenous elements $a,b,c\in \R$:
\begin{align}
&\a\partial=\partial\a, \b\partial=\partial \b, \\
&\a([a_\lambda b]) = [\a(a)_\lambda \a(b)], \b([a_\lambda b]) = [\b(a)_\lambda \b(b)],\\
&[\partial a_{\lambda}b] =-\lambda[a_\lambda b], [ a_{\lambda}\partial b] =(\partial+\lambda)[a_\lambda b],
\label{conformal sesq}\\
&[\b(a)_{\lambda}\a(b)]=-(-1)^{|a||b|}[\b(b)_{-\lambda-\partial}\a(a)],
\label{bihom skewsymm}\\
&[\a\b(a)_\lambda[b_\mu c]]=[[\b(a)_{\lambda}b]_{\lambda+\mu} \b(c)]+ (-1)^{|a||b|}[\b(b)_{\mu}[\a(a)_{\lambda}c]]. \label{bihom-jaccobi}
\end{align}

\end{df}
A BiHom-Lie conformal superalgebra $(\R,\a,\b)$ is called finite if $\R$ is a finitely generated
$\mathbb{C}[\pr]$-module. The rank of $\R$ is its rank as a $\mathbb{C}[\pr]$-module.
If $\a$ and $\b$ are  algebra isomorphisms, then  $(\R,\a,\b)$ is called regular.

We recover Lie conformal superalgebras when $\a=\b = id$. The BiHom-Lie conformal
algebras are obtained when the odd part is trivial.

A linear map $\rho: \R \to \R'$ is a homomorphism  of BiHom-Lie conformal superalgebras if $\rho$ satisfies $\rho \pr=\pr \rho$, $\rho \a =\a' \rho$,  $\rho \b=\b' \rho$ and
$\rho([a_\l b])=[\rho(a)_\l \rho(b)]$, for all $a,b  \in \R$.

\begin{exa}
Let $g=g_0\oplus g_1$ be a BiHom-Lie superalgebra with Lie bracket $[-,-]$ and with structure maps $\a$ and $\b$. Let $(Curg)_{\theta}:=\mathbb{C}[\partial]\o g_\theta$ be the free $\mathbb{C}[\partial]$-module. Then $Curg=(Curg)_0\oplus (Curg)_1$ is a BiHom-Lie conformal superalgebra with $\lambda$-bracket given by
\begin{eqnarray*}
&& \a(f(\partial)\o a)=f(\partial)\o \a(a),\\
&& \b(f(\partial)\o a)=f(\partial)\o \b(a),\\
&&[(f(\partial)\o a)_{\lambda}(g(\partial)\o b)]=f(-\lambda)g(\partial+\lambda)\o [a,b], \forall a,b\in g.
\end{eqnarray*}
\end{exa}

\begin{exa}
Let $(L, [\c,\c], \a, \b)$ be a regular BiHom-Lie superalgebra.
   Denote by $\hat{L}=L\o \mathbb{C}[t, t^{-1}]$ the affization of $L$ with
  \begin{eqnarray*}
 [u\o t^m, v\o t^n]=[u,v]\o t^{m+n},\ \ |u \otimes t^n|=|u|
  \end{eqnarray*}
  for any $u,v\in L, m,n\in \mathbb{Z}$.
  Extend $\a,\b$ to $\hat{L}$ by $\a(u\o t^{m})=\a(u)\o t^{m}$ and $\b(u\o t^{m})=\b(u)\o t^{m}$.
  Then $(\hat{L}, [\c,\c], \a, \b)$ is a BiHom-Lie superalgebra. By simple verification, we get a BiHom-Lie conformal superalgebra $\R=\mathbb{C}[\partial]L$ with
\begin{eqnarray*}
[u_{\lambda}v]=[u,v], ~~
\a(f(\partial)u)=f(\partial)\a(u), \b(f(\partial)u)=f(\partial)\b(u),
\end{eqnarray*}
 for any $u,v\in L.$
\end{exa}

\begin{exa}
Let $\R=\mathbb{C}[\partial]L\oplus \mathbb{C}[\partial]E$ be a free $\mathbb{Z}_2$-graded $\mathbb{C}[\partial]$-module. Define
\begin{eqnarray*}
&& \a(L)=f(\partial)L, \a(E)=g(\partial)E,\\
&& \b(L)=L, \b(E)=E, \\
&& [L_\lambda L]=(\partial+2\lambda)L, [L_\lambda E]=(\partial+\frac{3}{2}\lambda)E, [E_\lambda L]=(\frac{1}{2}\partial+\frac{3}{2}\lambda)E, [E_\lambda E]=0.
\end{eqnarray*}
Then $(\R,\a,\b)$ is a BiHom-Lie conformal superalgebra, where $R_0=\mathbb{C}[\partial]L$ and $ R_1= \mathbb{C}[\partial]E$.
\end{exa}

\begin{exa}
Let $\R=\mathbb{C}[\partial]e_1\oplus \mathbb{C}[\partial](e_2+e_3)$ be a free $\mathbb{Z}_2$-graded $\mathbb{C}[\partial]$-module and
 \[
 e_1=  \left(
   \begin{array}{ccc}
     0 & 0 & 1 \\
     0 & 0 & 0 \\
     0 & 0 & 0 \\
   \end{array}
 \right),
  e_2= \left (
   \begin{array}{ccc}
     0 & 1 & 0 \\
     0 & 0 & 0 \\
     0 & 0 & 0 \\
   \end{array}
 \right),
 e_3= \left(
   \begin{array}{ccc}
     0 & 0 & 0 \\
     0 & 0 & 1 \\
     0 & 0 & 0 \\
   \end{array}
 \right).
 \]
    Define\begin{eqnarray*}
       && \a(e_1)=e_1, \a(e_2)=e_3, \a(e_3)=e_2,\\
       &&  \b(e_1)=e_1, \b(e_2)=-e_3, \b(e_3)=-e_2,\\
&& [e_{1\lambda} e_1]= [e_{2\lambda} e_2]=[e_{3\lambda} e_3]=0, [e_{1\lambda} e_2]=[e_{1\lambda} e_3]=0, [e_{2\lambda}e_3]= e_1.
          \end{eqnarray*}
One may check directly that $(\R,\a,\b)$ is a BiHom-Lie conformal superalgebra.
\end{exa}

\begin{prop}
Let $\R$ be a Lie conformal superalgebra and $\a,\b:\R \rightarrow \R$ two even commuting linear maps such that
 $$\a\partial=\partial\a, \b\partial=\partial\b, \a([a_\lambda b]) = [\a(a)_\lambda \a(b)], \b([a_\lambda b]) = [\b(a)_\lambda \b(b)].$$
Define   a $\mathbb{C}$-linear map from $\R \otimes  \R$ to  $\R[\lambda]=\mathbb{C}[\lambda]\otimes \R$ by $[a_{\lambda}b]'=[\a(a)_{\lambda}\b(b)]$.
Then $(\R,[\c_\l  \c]',\a,\b)$ is a BiHom-Lie conformal superalgebra.
\end{prop}

\begin{proof}
We will just verify axioms \eqref{bihom skewsymm} and \eqref{bihom-jaccobi}. Let $a,b,c \in \R$. Then
\begin{align*}
 [\b(a)_\l \a(b)]'=[\a\b(a)_\l \a\b(b)]
 = -[\a\b(b)_{-\l-\pr}\a\b(a)]
 = -[\b(b)_{-\l-\pr} \a(a)]'.
\end{align*}
On the other hand, we have
\begin{align*}
&[\a\b(a)_\l[b_\mu c]']'= [\a^2\b(a)_\l [\a\b(b)_\mu \b^2(c)]]  \\
=&[ [\a^2\b(a)_\l \a\b(b)]_{\l+\mu} \b^2(c)]+(-1)^{|a||b|}[\a\b(b)_\mu[\a^2\b(a)_\l \b^2(c)]] \\
=&[[\b(a)_{\lambda}b]'_{\lambda+\mu} \b(c)]'+ (-1)^{|a||b|}[\b(b)_{\mu}[\a(a)_{\lambda}c]']'.
\end{align*}
\end{proof}

The following result can be proved similarly to the previous one.
\begin{prop}
Let $(\R,[\c_\l \c],\a,\b)$  be a BiHom-Lie conformal superalgebra.   Let $\a',\b':  \R \to \R$ be two even homomorphisms of conformal algebras such that any two of the maps $\a,\b,\a',\b'$ commute.  Then $(\R,[\c_\l\c]'=[\c_\l \c]\circ (\a' \otimes \b'), \a \circ \a',\b \circ \b')$ is a BiHom-Lie conformal superalgebra.
\end{prop}

\begin{cor}
Let $(\R,[\c_\l \c],\a,\b)$  be a BiHom-Lie conformal superalgebra. Then $(\R,[\c_\l\c]'=[\c_\l \c]\circ (\a^k \otimes \b^k), \a^{k+1},\b^{k+1})$ is a BiHom-Lie conformal superalgebra.
\end{cor}

\begin{prop}
Let $(\R,\a,\b)$ and  $(\R',\a',\b')$ be two BiHom-Lie conformal superalgebras,   Then $(\R\oplus \R',\a+\a',\b+\b')$ is a BiHom-Lie conformal superalgebra with
\begin{eqnarray*}
&&[(u_1+v_1)_\lambda (u_2+v_2)]=[u_{1 \lambda}v_1]+[u_{2 \lambda}v_2],\forall u_1,u_2\in \R, \forall v_1,v_2\in \R',\\
&&(\a+\a')(u+v)=\a(u)+\a'(v),    (\b+\b')(u+v)=\b(u)+\b'(v), \forall u\in \R, \forall v\in \R'.
\end{eqnarray*}
\end{prop}

\begin{proof}
Straightforward.
\end{proof}

\begin{df}
Let $(M,\a, \b)$ and $(N, \f,\p)$ be $\mathbb{Z}_2$-graded $\mathbb{C}[\partial]$-modules. A BiHom-conformal linear map of degree $\theta$ from $M$ to $N$ is a sequence $f={f_{(n)}}_{n\in \mathbb{Z}_{+}}$ of $f_{(n)}\in Hom_{\mathbb{C}}(M, N)$ satisfying that
\begin{eqnarray*}
&&\partial_N f_{(n)}-f_{(n)}\partial_M=-n f_{(n-1)}, f_{\lambda}(M_\mu)\subseteq N_{\mu+\theta}, n\in \mathbb{Z}_{+}, \mu, \theta\in \mathbb{Z}_2,\\
&& \partial_M\a=\a \partial_M,    \partial_M\b=\b \partial_M,
\partial_N \f=\f \partial_N,  \partial_N \p=\p \partial_N,  \\
&& f_{(n)}\a=\f f_{(n)}, f_{(n)}\b=\p f_{(n)}.
\end{eqnarray*}
Set $f_\lambda=\sum_{n=0}^{\infty}\frac{\lambda^{n}}{n!}f_{(n)}$. Then $f$ is a BiHom-conformal linear map of degree $\theta$ if and only if
\begin{eqnarray*}
&&f_{\lambda} \partial_M=(\partial_N+\lambda)f_\lambda, f_{\lambda}(M_\mu)\subseteq N_{\mu+\theta}[\lambda],\\
&& \partial_M\a=\a \partial_M,    \partial_M\b=\b \partial_M,
\partial_N \f=\f \partial_N,  \partial_N \p=\p \partial_N,  \\
&&f_{\lambda}\a=\f f_{\lambda},f_{\lambda}\b=\p f_{\lambda}.
\end{eqnarray*}
\end{df}

Let $Chom(M,N)_{\theta}$ denote the set of BiHom-conformal linear maps of degree $\theta$ from $M$ to $N$. Then $Chom(M,N)=Chom(M,N)_0 \oplus Chom(M,N)_1$ is a $\mathbb{Z}_{2}$-graded $\mathbb{C}[\partial]$-module via:
\begin{eqnarray*}
\partial f_{(n)}=-n f_{(n-1)},\ {\rm equivalently},\ \partial f_{\lambda}=-\lambda f_{\lambda}.
\end{eqnarray*}
The composition $f_{\lambda}g: M \rightarrow N\otimes\mathbb{C}[\lambda]$ of BiHom-conformal linear maps $f:M\rightarrow N$ and $g:L\rightarrow M$ is given by
 \begin{eqnarray*}
(f_{\lambda}g)_{\lambda+\mu}=f_{\lambda}g_{\mu},\ \ \ \forall  f,g\in Chom(M,N).
 \end{eqnarray*}

If $M$ is a finitely generated $\mathbb{Z}_{2}$-graded $\mathbb{C}[\partial]$-module, then $Cend(M):=Chom(M,M)$ is an associative conformal superalgebra with respect to the above composition. Thus, $Cend(M)$ becomes a Lie conformal superalgebra, denoted as $gc(M)$, with respect to the following $\lambda$-bracket
\begin{eqnarray}\label{crochet de cend}
[f_{\lambda}g]_{\mu}=f_{\lambda}g_{\mu-\lambda}-(-1)^{|f||g|}g_{\mu-\lambda}f_{\lambda}.
\end{eqnarray}
Hereafter all $\mathbb{Z}_{2}$-graded $\mathbb{C}[\partial]$-modules are supposed to be  finitely generated.

\begin{prop}
Let $(\R,\a,\b)$ be a regular BiHom-associative conformal superalgebra with two even morphisms $\a$ and $\b$. Then the $\l$-bracket
 \begin{eqnarray*}
[a_{\lambda}b]=a_{\lambda} b-(-1)^{|a||b|}\a^{-1}\b(b)_{-\lambda-\partial} \a\b^{-1}(a), \forall a,b\in \R.
  \end{eqnarray*}
defines a  BiHom-Lie conformal  superalgebra structure on $\R$.
\end{prop}

\begin{proof}
For any $a,b \in \R$, we have
\begin{align*}
    &[\b(a)_\l \a(b)]=\b(a)_\l \a(b)-(-1)^{|a||b|}\b(b)_{-\lambda-\partial}\a(a) \\
    =&-(-1)^{|a||b|}\big(\b(b)_{-\lambda-\partial}\a(a)-(-1)^{|a||b|}\b(a)_\l \a(b) \big) \\
   =&  -(-1)^{|a||b|} [\b(b)_\l \a(a)].
\end{align*}
On the other hand, given $a,b,c \in \R$. Then
\begin{eqnarray*}
&&[\a\b(a)_{\lambda}[b_{\mu}c]]
=[\a\b(a)_{\lambda}(b_{\mu}c-(-1)^{|b||c|}\a^{-1}\b(c)_{-\mu-\partial} \a\b^{-1}(b))]\\
&=&\a\b(a)_{\lambda}(b_{\mu}c)-(-1)^{|a|(|b|+|c|)}a^{-1}\b(b_{\mu}c)_{-\lambda-\partial}\a^2(a)
-(-1)^{|b||c|}\a\b(a)_{\lambda}(\a^{-1}\b(c)_{-\mu-\partial} \a\b^{-1}(b)) \\
&+&(-1)^{|a||c| +|a||b|+|b||c|}\a^{-1}\b(\a^{-2}\b^2(c)_{-\mu-\partial} b)_{-\lambda-\partial}\a^2(a).
  \end{eqnarray*}
Similarly, we have
\begin{eqnarray*}
&&[[\b(a)_\lambda b]_{\lambda+\mu}\b(c)]\\
&=& [(\b(a)_{\lambda} b-(-1)^{|a||b|}\a^{-1}\b(b)_{-\lambda-\partial} \a(a))_{_{\lambda+\mu}}\b(c)]\\
&=& (\b(a)_{\lambda}b)_{\lambda+\mu}\b(c)
-(-1)^{|c|(|a|+|b|)}\a^{-1}\b^2(c)_{-\lambda-\mu-\partial}(\a(a)_{\lambda} \a\b^{-1}(b))\\
&&-(-1)^{|a||b|}(\a^{-1}\b(b)_{-\lambda-\partial} \a(a))_{_{\lambda+\mu}}\b(c)+(-1)^{|a||c|+
|a||b|+|b||c|}\a^{-1}\b^2(c)_{-\lambda-\mu-\partial}(b_{-\lambda-\partial} \a^2\b^{-1}(a))
\end{eqnarray*}
and
\begin{eqnarray*}
&&(-1)^{|a||b|}[\b(b)_{\mu},[\a(a)_{\lambda}c]]\\
&=&(-1)^{|a||b|}\b(b)_{\mu}(\a(a)_{\lambda}c)
-(-1)^{|a|(|c|+|b|)}\b(b)_{\mu}(\a^{-1}\b(c)_{-\lambda-\partial} \a^2\b^{-1}(a))\\
&&-(-1)^{|a||c|}(\b(a)_{\lambda}\a^{-1}\b(c))_{-\mu-\partial}\a(b)
+(-1)^{|a||b|}(\a^{-2}\b^2(c)_{-\lambda-\partial} \a(a))_{-\mu-\partial}\a(b).
  \end{eqnarray*}
Using the associativity, it is not hard to check that
\begin{eqnarray*}
[\a\b(a)_\lambda[b_\mu c]]=[[\b(a)_{\lambda}b]_{\lambda+\mu} \b(c)]+ (-1)^{|a||b|}[\b(b)_{\mu}[\a(a)_{\lambda}c]]
\end{eqnarray*}
as desired. This finishes the proof.
\end{proof}

\section{Cohomology of BiHom-Lie conformal superalgebra}
In this section, we develop representation and cohomology theory of BiHom-Lie conformal superalgebras. 
\begin{df}
A $\mathbb{C}[\pr]$-module $(M,\f,\p)$ is a BiHom-conformal module of
a BiHom-Lie conformal superalgebra $\R$ if there is a $\mathbb{C}$-linear map $\rho: \R \to Cend(M)$ satisfying
the following conditions (for all $a,b \in \R$)
\begin{eqnarray}
&&\rho(\partial(a))_{\lambda}=-\lambda\rho(a)_{\lambda},
\rho(a)_\l \pr=(\l+\pr)\rho(a)_\l \\
&&   \f \p=\p \f,  \
\p  \partial =\partial  \p,\  \f  \partial =\partial  \f,\\
&&\f \rho(a)_\lambda =\rho(\a(a))_\lambda \f ,\p \rho(a)_\lambda =\rho(\b(a))_\lambda \p,\\
&&     \rho([\b(a)_\lambda b])_{\lambda+\mu} \p=\rho(\a\b(a))_\l  \rho(b)_\mu-(-1)^{|a||b|} \rho(\b(b))_\mu  \rho(\a(a))_\l.
\end{eqnarray}
The map $\rho$ is called the corresponding representation.
\end{df}

\begin{exa}
Let  $(\R,\a,\b)$  be a BiHom-Lie conformal superalgebra. Then $(\R,\a,\b)$ is an $\R$-module under the
adjoint diagonal action, $\rho(a)_\l b=[a_\l b],\ \forall a,b \in \R$.
\end{exa}

\begin{exa}
Let $(g,[\c,\c],\a,\b)$ be a finite BiHom-Lie superalgebra, $(M,\f,\p)$ be a BiHom-conformal module of $g$ with corresponding representation map  $\varrho: g \to gl(M)$.
Then the $\mathbb{Z}_2$-graded $\mathbb{C}[\pr]$-module $(\mathbb{C}[\pr]\otimes M,\widetilde{\f},\widetilde{\p})$ is a BiHom-conformal module of $Cur\ g$
where
\begin{align*}
&\widetilde{\f}(f(\pr)\otimes m)=f(\pr)\otimes \f(m),\  \widetilde{\p}(f(\pr)\otimes m)=f(\pr)\otimes \p(m),
\end{align*}
 with a a module structure $\rho: Cur \ g \to Cend(\mathbb{C}[\pr]\otimes M)$ given by
\begin{align*}
& \rho(f(\pr)\otimes a)_\l (g(\pr)\otimes m)=f(-\pr)g(\pr+\l)\otimes \varrho(a)(m),\
\forall f,g \in \mathbb{C}[\pr], a \in g, m \in M,
\end{align*}
\end{exa}

\begin{prop}
Let $(\R,\a,\b)$ be a regular BiHom-Lie conformal superalgebra, $(M,\f,\p)$ a BiHom-conformal module of $\R$ and $\rho:\R\rightarrow\mathbb{C}[\lambda]\otimes Cend(M): a \mapsto\rho(a)_\lambda$ the corresponding representation. Define a $\lambda$-bracket $[-_\lambda -]_M$ on $\R \oplus M=(\R \oplus M)_0\oplus(\R \oplus M)_1$ by
\begin{eqnarray}\label{lemm2-14}
[(r+m)_\lambda(r'+m')]_M=[r_\lambda r']+\rho(r)_\lambda m'-(-1)^{|r'||m|}\rho(\a^{-1}\b(r'))_{-\partial-\lambda}\f\p^{-1}(m),
\end{eqnarray}
$\forall \, r+m,r'+m'\in\R\oplus M$, where $(\R \oplus M)_{\theta}=\R_\theta \oplus M_\theta, \theta\in\mathbb{Z}_{2}.$
Then $(\R\oplus M, \a+\f,\b+\p)$ is a BiHom-Lie conformal superalgebra, called the semidirect product of $\R$ and $M$, and denote by $\R\ltimes_{\rho} M$.
\end{prop}

\begin{proof}
$\forall\, r+m,r'+m',r''+m''\in \R\oplus M$, note that $\R \oplus M$ is equipped with a $\mathbb{C}[\partial]$-module structure via
\begin{eqnarray*}
\partial(r+m)=\partial(r)+\partial(m).
\end{eqnarray*}
In addition, it is immediate that $\pr(\a+\f)=(\a+\f)\pr$ and  $\pr(\b+\p)=(\b+\p)\pr$.
A direct computation gives
\begin{eqnarray*}
&&[\partial(r+m)_\lambda(r'+m')]_M
=[(\partial r+\partial m)_\lambda(r'+m')]_M\\
&=&[(\partial r)_\lambda r']+\rho(\partial r)_\lambda(m')-(-1)^{|r'||m|}\rho(\a^{-1}\b(r'))_{-\partial-\lambda}(\partial \f\p^{-1}(m))\\
&=&-\lambda[r_\lambda r']-\lambda \rho(r)_\lambda(m')-
(-1)^{|r'||m|}(\partial-\lambda-\partial)\rho(\a^{-1}\b(r'))_{-\partial-\lambda}(\f\p^{-1}(m))\\
&=&-\lambda([r_\lambda r']+\rho(r)_\lambda(m')-(-1)^{|r'||m|}\rho(\a^{-1}\b(r'))_{-\partial-\lambda}(\f\p^{-1}(m)))\\
&=&-\lambda[(r+m)_\lambda(r'+m')]_M
\end{eqnarray*}
and
\begin{align*}
&[(r+m)_\l \pr(r'+m')]_M=[(r+m)_\l(\pr r'+\pr m')]_M \\
=&[r_\l \pr r']+\rho(r)_\l (\pr m')-(-1)^{|r'||m|}\rho (\pr \a^{-1}\b(r'))_{-\l-\pr} \f\p^{-1}(m) \\
=&(\l+\pr)([r_\l r'] +  \rho(r)_\l (m')-(-1)^{|r'||m|}\rho (\a^{-1}\b(r'))_{-\l-\pr} \f\p^{-1}(m) )\\
=& (\l+\pr) [(r+m)_\l (r'+m')]_M.
\end{align*}
Thus \eqref{conformal sesq} holds. \eqref{bihom skewsymm} follows from
\begin{align*}
&[(\b(r)+\p(m))_\lambda(\a(r')+\f(m'))]_M  \\
=& [\b(r)_\l \a(r')]+\rho(\b(r))_\l(\f(m'))-(-1)^{|r'||m|}\rho(\b(r'))_{-\pr-\l}(\f(m))\\
=& -(-1)^{|r'||m|}([\b(r')_{-\partial-\lambda}\a(r)]+\rho(\b(r'))_{-\partial-\lambda}(\f(m))
-(-1)^{|r||m|}\rho(\b(r))_{-\pr-(-\pr-\l)}(\f(m'))\\
=&  -(-1)^{|r'||m|} [(\b(r')+\p(m'))_{-\partial-\lambda}(\a(r)+\f(m))]_M.
\end{align*}
To check the BiHom-Jacobi conformal identity \eqref{bihom-jaccobi} , we compute
\begin{eqnarray}
&&[(\a\b(r)+\f\p(m))_\lambda[(r'+m')_\mu(r''+m'')]_M]_M\nonumber\\
&=&[(\a\b(r)+\f\p(m))_\lambda([r'_\mu r'']+\rho(r')_\mu(m'')-(-1)^{|r''||m'|}\rho(\a^{-1}\b(r''))_{-\partial-\mu}(\f\p^{-1}(m'))]_M\nonumber\\
&=&[\a\b(r)_\l[r'_\mu r'']]+\rho(\a\b(r))_\l\rho(r')_\mu(m'')-
(-1)^{|r''||m'|}\rho(\a\b(r))_\l\rho(\a^{-1}\b(r''))_{-\pr-\mu}\f\p^{-1}(m')\nonumber\\
&&-(-1)^{(|r''|+|r'|)|m|}\rho([\a^{-1}\b(r')_\mu \a^{-1}\b(r'')])_{-\partial-\lambda}\f^2(m),\label{1}\\[4pt]
&&(-1)^{|r||r'|}[(\b(r')+\p(m'))_\mu[(\a(r)+\f(m))_\lambda(r''+m'')]_M]_M\nonumber\\
&=&(-1)^{|r||r'|}[\b(r')_\mu[\a(r)_\l r'']]+(-1)^{|r||r'|}\rho(\b(r'))_\mu(\rho(\a(r))_\l(m''))\nonumber\\
&&-(-1)^{|r||r'|+|r''||m|}\rho(\b(r'))_\mu(\rho(\a^{-1}\b(r''))_{-\partial-\l}(\f^2\p^{-1}(m))) \nonumber\\
&&-(-1)^{|r''||m'|}\rho([\b(r)_\lambda \a^{-1}\b(r'')])_{-\partial-\mu}\f(m')\label{2}
\end{eqnarray}
and
\begin{eqnarray}
&&[{[(\b(r)+\p(m))_\l(r'+m')]_M}_{(\l+\mu)}(\b(r'')+\p(m''))]_M\nonumber\\
&=&[([\b(r)_\l r']+\rho(\b(r))_\l(m')-(-1)^{|r'||m|}\rho(\a^{-1}\b(r'))_{-\pr-\l} \f\p^{-1}(m))_{\l+\mu}(\b(r'')+\p(m''))]_M\nonumber\\
&=&[[\b(r)_\lambda \b(r')]_{\l+\mu}r'']+\rho([\b(r)_\l r'])_{\lambda+\mu}(\p(m'')) \nonumber\\
&&\quad  -(-1)^{|r''|(|r|+|m'|)}\rho(\a^{-1}\b^2(r''))_{-\pr-\l-\mu}\f\p^{-1}(\rho(\b(r))_\l m')\nonumber\\
&&+(-1)^{|r''|(|r'|+|m|)+|r'||m|}\rho(\a^{-1}\b^2(r''))_{-\pr-\l-\mu}\f\p^{-1}(\rho(\a^{-1}\b(r'))_{-\pr-\l}\f(m)).\label{3}
\end{eqnarray}
Now, since $(M,\rho,\f,\p)$ is  a representation of $\R$, then we can easily conclude that
$\R\oplus M$ is a BiHom-Lie conformal superalgebra.
\end{proof}

In the following we aim to develop the cohomology theory of regular BiHom-Lie con-
formal superalgebras. To do this, we need the following concept.

\begin{df}
An $n$-cochain ($n\in \mathbb{Z}_+$) of a BiHom-Lie conformal superalgebra $\R$ with coefficients in a module $M$ is a $\mathbb{C}$-linear map of degree $\theta$
\begin{eqnarray*}
\gamma:\R^{n}&\rightarrow& M[\lambda_{1},\cdots,\lambda_{n}],\\
(a_1,\cdots,a_n)&\mapsto& \gamma_{\lambda_{1},\cdots,\lambda_{n}}(a_1,\cdots,a_n),
\end{eqnarray*}
where $M[\lambda_{1},\cdots,\lambda_{n}]$ denotes the space of polynomials with coefficients in $M$, satisfying the following conditions:

Conformal antilinearity:
\begin{eqnarray*}
\gamma_{\l_{1},\cdots,\l_{n}}(a_1,\cdots,\partial a_i,\cdots,a_n) =-\l_{i}\gamma_{\l_{1},\cdots,\l_{n}}(a_1,\cdots,a_i,\cdots,a_n).
\end{eqnarray*}

Skew-symmetry:
\begin{eqnarray*}
&&\gamma_{\l_{1},\cdots,\l_i,\l_{i+1},\cdots,\l_{n}}(a_1,\cdots,\b(a_i),\a(a_{i+1}),\cdots,a_n)\\
&=&-(-1)^{|a_i||a_{i+1}|}\g_{\l_{1},\cdots,\l_{i+1},\l_i,\cdots,\l_{n}}(a_1,\cdots,\b(a_{i+1}),\a(a_i),\cdots,a_n).
\end{eqnarray*}

Commutativity:  $\g \circ \a= \f \circ \g,\  \g \circ \b = \p \circ \g$.
\end{df}
A $0$-cochain is just an element in $M$. Define a differential $d$ of a cochain $\g$
 by
\begin{eqnarray*}
&&(d \gamma)_{\l_{1},\cdots,\l_{n+1}}(a_1,\cdots,a_{n+1})\\
&=&\mbox{$\sum\limits_{i=1}^{n+1}$}(-1)^{i+1}
(-1)^{(|\gamma|+|a_1|+\cdots+|a_{i-1}|)|a_i|}\rho(\a\b^{n-1}(a_i))_{\l_i}\gamma_{\l_{1},\cdots,\hat{\lambda_i},
\cdots,\l_{n+1}}(a_1,\cdots,\hat{a_i},\cdots,a_{n+1})\\
&&+\mbox{$\sum\limits_{1\leq i<j}^{n+1}$}(-1)^{i+j}(-1)^{(|a_1|+\cdots+|a_{i-1}|)|a_i|+(|a_1|+\cdots+|a_{j-1}|)|a_j|+|a_i||a_j|}\\
&&\gamma_{\lambda_i+\lambda_j,\lambda_1,\cdots,\hat{\lambda_i},\cdots,\hat{\lambda}_j,\cdots,\lambda_{n+1}}
([\a^{-1}\b(a_i)_{\lambda_i}a_j],\b(a_1),\cdots,\hat{a}_i,\cdots,\hat{a}_j,\cdots,\b(+a_{n+1})),
\end{eqnarray*}
where $\rho$ is the corresponding representation of $M$, and $\gamma$ is extended linearly over the polynomials in $\lambda_i$. In particular, if $\gamma$ is a $0$-cochain, then $(d \gamma)_\lambda a=a_\lambda \gamma$.

\begin{prop}
 $d \gamma$ is a cochain and $d^2=0$.
\end{prop}

\begin{proof}

Let $\g$ be an $n$-cochain.  It is not hard to see that $d\g$ satisfies conformal antilinearity, skew-symmetry  and commutativity.
  That is, $d\g$ is an $(n+1)$-cochain. It remains to show that $d^2=0$.  In fact, we have
 \begin{align}
& (d^2\g) _{\lambda_1,...,\lambda_{n+2}}(a_{1},...,a_{n+2})\nonumber\\
=& \sum_{i=1}^{ n+1}(-1)^{i+1}(-1)^{i+1+|\gamma||a_i|+A_i}\rho(\alpha\b^{n}(a_{i}))_{\lambda_i}(d\g)_{\lambda_1,...,\hat{\lambda}_{i},...,\lambda_{n+2}}(a_{1},..., \hat{a}_{i},..., a_{n+2})\nonumber\\
  &+\sum_{1\leq i<j\leq n+1} (-1)^{i+j+A_i+A_j+|a_i||a_j|}(d\g)_{\lambda_i+\lambda_j, \lambda_1,...,\hat{\lambda}_i,...\hat{\lambda}_j,...,\lambda_{n+2}} \nonumber \\
  & \hspace{4 cm}([\a^{-1}\b(a_{i})_{\lambda_i}a_{j}], \b(a_{1}),... ,\hat{a}_{i},... , \hat{a}_{j},...,\b(a_{n+2}))\nonumber\\
  =&\sum_{i=1}^{n+2}\sum_{j=1}^{i-1}(-1)^{i+j+|\gamma|(|a_i|+|a_j|)+A_i+A_j}
  \rho(\alpha\b^{n}(a_{i}))_{\lambda_i}(\alpha\b^{n-1}(a_{j})_{\lambda_j} \nonumber \\
 &\hspace{4 cm}  \g_{\lambda_1,...,\hat{\lambda}_{j,i},...,\lambda_{n+2}}(a_{1},..., \hat{a}_{j,i},..., a_{n+2})\\
  &\sum_{i=1}^{n+2}\sum_{j=i+1}^{n+2}(-1)^{i+j+1+|\gamma|(|a_i|+|a_j|)+A_i+(A_j-|a_i|)}
  \rho(\alpha\b^{n}(a_{i}))_{\lambda_i} \nonumber \\
 & \hspace{4 cm} (\alpha\b^{n-1}(a_{j})_{\lambda_j}\g_{\lambda_1,...,\hat{\lambda}_{i,j},...,\lambda_{n+2}}(a_{1},..., \hat{a}_{i,j},..., a_{n+2})\\
  &+\sum_{i=1}^{n+2}\sum_{1\leq j< k<i}^{n+2}(-1)^{i+j+k+1+|\gamma||a_i|+A_i+A_j+A_k+|a_j||a_k|}
  \rho(\alpha\b^{n}(a_{i}))_{\lambda_i}\g_{\lambda_j+\lambda_k, \lambda_1,...,\hat{\lambda}_{j,k,i},...,\lambda_{n+2}}\nonumber\\
  &([\a^{-1}\b(a_{j})_{\lambda_j}a_{k}], \b(a_{1}),... ,\hat{a}_{j,k,i},...,\b(a_{n+2}))\\
  &+\sum_{i=1}^{n+2}\sum_{1\leq j< i<k}^{n+2}(-1)^{i+j+k+|\gamma||a_i|+A_i+A_j+(A_k-|a_i|)+|a_j||a_k|}
  \rho(\alpha\b^{n}(a_{i}))_{\lambda_i}\g_{\lambda_j+\lambda_k, \lambda_1,...,\hat{\lambda}_{j,i,k},...,\lambda_{n+2}}\nonumber\\
  &([\a^{-1}\b(a_{j})_{\lambda_j}a_{k}], \b(a_{1}),... ,\hat{a}_{j,i,k},...,\b(a_{n+2}))\\
   &+\sum_{i=1}^{n+2}\sum_{1\leq i< j<k}^{n+2}(-1)^{i+j+k+1+|\gamma||a_i|+A_i+(A_j-|a_i|)+(A_k-|a_i|)+|a_j||a_k|}
   \rho(\alpha\b^{n}(a_{i}))_{\lambda_i}\g_{\lambda_j+\lambda_k, \lambda_1,...,\hat{\lambda}_{i,j,k},...,\lambda_{n+2}}\nonumber\\
  &([\a^{-1}\b(a_{j})_{\lambda_j}a_{k}], \b(a_{1}),... ,\hat{a}_{i,j,k},...,\b(a_{n+2})) \\
  &+\sum_{1\leq i< j}^{n+2}\sum_{k=1}^{i-1}(-1)^{i+j+k+A_i+A_j+A_k+|a_i||a_j|+(|\gamma|+|a_i|+|a_j|)|a_k|}
  \rho(\alpha\b^{n}(a_{k}))_{\lambda_k}\g_{\lambda_i+\lambda_j, \lambda_1,...,\hat{\lambda}_{k, i,j},...,\lambda_{n+2}}\nonumber\\
  &([\a^{-1}\b(a_{i})_{\lambda_i}a_{j}], \b(a_{1}),... ,\hat{a}_{k, i,j},...,\b(a_{n+2}))\\
  &+\sum_{1\leq i< j}^{n+2}\sum_{k=i+1}^{j-1}(-1)^{i+j+k+1+A_i+A_j+A_k+|a_i||a_j|+(|\gamma|+|a_j|)|a_k|}
  \rho(\alpha\b^{n}(a_{k}))_{\lambda_k}\g_{\lambda_i+\lambda_j, \lambda_1,...,\hat{\lambda}_{i, k,j},...,\lambda_{n+2}}\nonumber\\
  &([\a^{-1}\b(a_{i})_{\lambda_j}a_{j}], \b(a_{1}),... ,\hat{a}_{i, k,j},...,\b(a_{n+2}))\\
  &+\sum_{1\leq i< j}^{n+2}\sum_{k=j+1}^{n+2}(-1)^{i+j+k+A_i+A_j+A_k+|a_i||a_j|+|\gamma||a_k|}
  \rho(\alpha\b^{n}(a_{k}))_{\lambda_k}\g_{\lambda_i+\lambda_j,  \lambda_1,...,\hat{\lambda}_{i, j,k},...,\lambda_{n+2}}\nonumber\\
  &([\a^{-1}\b(a_{i})_{\lambda_i}a_{j}], \b(a_{1}),... ,\hat{a}_{i,j,k},...,\b(a_{n+2}))\\
  &+\sum_{1\leq i< j}^{n+2}(-1)^{i+j+A_i+A_j+|a_i||a_j|+|\gamma|(|a_i|+|a_j|)}
  \rho(\alpha\b^{n-1}([\a^{-1}\b(a_{i})_{\lambda_i}a_{j}]))_{\lambda_i+\lambda_j} \nonumber \\
  & \hspace{4 cm}\g_{\lambda_1,...,\hat{\lambda}_{j},...\hat{\lambda}_{i},...,
  \lambda_{n+2}}(\b(a_{1}),... ,\hat{a}_{j},...,\hat{a}_{i},...,\b(a_{n+2})~~~~~~~~~\\
  &+\sum^{n+2}_{\textrm{distinct} i,j,k,l,i< j,k<l}(-1)^{i+j+k+l}\textrm{sg}\{i,j, k, l\}
  (-1)^{A_i+A_j+|a_i||a_j|+(|a_i|+|a_j|)(|a_k|+|a_l|)}\nonumber\\
  &\g_{\lambda_k+\lambda_l, \lambda_i+\lambda_j,\lambda_1,...,\hat{\lambda}_{i,j,k,l},...,\lambda_{n+2}}(\b[\a^{-1}\b(a_{k})_{\lambda_k}a_{l}],
  \b[\a^{-1}\b(a_{i})_{\lambda_i}a_{j}],...,\hat{a}_{i,j,k,l},..., \b^2(a_{n+2}))~~~~~~~~~\\
   &\sum^{n+2}_{i,j,k=1,i< j,k\neq i,j}(-1)^{i+j+k+l}\textrm{sg}\{i,j, k\}
 (-1)^{A_i+A_j+|a_i||a_j|+(A_k-|a_i|-|a_j|)} \nonumber\\
 &\g_{\lambda_i+\lambda_k+\lambda_j,\lambda_1,...,\hat{\lambda}_{i,j,k},...,\lambda_{n+2}}([[\a^{-1}\b(a_i)_{\lambda_i}a_{j}]_{\lambda_i+\lambda_j},\b(\a_k)],
  \b^2(a_1),...,\hat{a}_{i,j,k},..., \b^2(a_{n+2})),
 \end{align}
where $\textrm{sg}\{i_1, ... , i_p\}$ is the sign of the permutation putting the indices in increasing
order and $\hat{a}_{i,j}$, $a_i, a_j,... $ are omitted.
\medskip

It is obvious that (3.17) and (3.22) summations cancel each other. The same is true
for (3.18) and (3.21), (3.19) and (3.20). The BiHom-Jacobi identity implies (3.25) = 0 and
the skew-symmetry implies  (3.24)=0. Because $(M, \rho,\f,\p)$ is an $\R$-module, it follows that
\begin{eqnarray*}
  \rho([\b(a)_\lambda b])_{\lambda+\mu} \p(v)=\rho(\a\b(a))_\l  \rho(b)_\mu(v)-(-1)^{|a||b|} \rho(\b(b))_\mu  \rho(\a(a))_\l(v).
\end{eqnarray*}
Since $\g\circ \a=\a_M\circ \g, \g\circ \b=\b_M\circ \g$, we have (3.15), (3.16) and (3.23) summations vanish.
Therefore $d^2\g=0$ and the proof is completed.

\end{proof}

Thus the cochains of a BiHom-Lie conformal superalgebra $(\R,\a,\b)$ with coefficients in a module $M$ form a comlex, which is denoted by
 $$ C^{\bullet}_{\a,\b}=C^{\bullet}_{\a,\b}(\R,M)=\bigoplus_{n\in \mathbb{Z}_{+}}C^{n}_{\a,\b}(\R,M).$$ Where $ C^{n}_{\a,\b}(\R,M)=C^{n}_{\a,\b}(\R,M)_0\oplus C^{n}_{\a,\b}(\R,M)_1$, $C^{n}_{\a,\b}(\R,M)_{\theta}$ is the set of $n$-cochain of degree $\theta$ $(\theta\in\mathbb{Z}_{2})$.
The cohomology group ${ H}^{\bullet}_{\a,\b}(\R,M)$ of a BiHom-Lie conformal superalgebra $\R$ with coefficients
in a module $M$ is the cohomology of complex $ C^{\bullet}_{\a,\b}$.

\begin{df}
Let $(\R,\a,\b)$ be a regular BiHom-Lie conformal superalgebra , $(M,\f,\p)$ a BiHom-conformal module of $\R$ and $\rho:\R\rightarrow\mathbb{C}[\lambda]\otimes Cend(M): a \mapsto\rho(a)_\lambda$ the corresponding representation.  If a $\mathcal{C}[\pr]$-module homomorphism $T: M \to A$  satisfies
\begin{align}\label{O-operator}
& T\circ   \f= \a \circ T,\quad T \circ \p=\b \circ T, \\
 &   [Tu_\l Tv]=T(\rho(Tu)_\l v-(-1)^{|u||v|}\rho(T\f^{-1}\p(v))_{-\l-\pr}\f\p^{-1}(u)),\ \forall u,v \in M,
\end{align}
then $T$ is called an $\mathcal{O}$-operator associated with $\rho$.
\end{df}

\begin{prop}
Let $(\R,\a,\b)$ be a regular BiHom-Lie conformal superalgebra,  $(M,\rho,\f,\p)$ be a representation and $T: M \to A$ be an $\mathcal{O}$-operator. Then the $\l$-product
\begin{align}\label{crochet sur V}
 [u_\l v]_T=\rho(Tu)_\l v-(-1)^{|u||v|}\rho(T\f^{-1}\p(v))_{-\l-\pr}\f\p^{-1}(u),\ u,v \in M
\end{align}
defines a BiHom-Lie superalgebra structure on $M$.
\end{prop}

\begin{proof}
It can be checked by a direct computation, so we omit details.
\end{proof}

The following result is obvious. 
\begin{cor}
Let $T$ be an $\mathcal{O}$-operator on a BiHom-Lie superalgebra  $(\R,\a,\b)$ with respect to a representation $(M,\rho,\f,\p)$. Then  $T$ is a homomorphism  from the
BiHom-Lie superalgebra $(M,[\c_\l \c]_T,\f,\p)$ to the BiHom-Lie superalgebra $(\R,\a,\b)$.
\end{cor}

\section{Generalized derivations of BiHom-Lie conformal superalgebras }
In this section, we investigate derivations and generalized derivations of BiHom-Lie conformal superalgebras and study some of their properties. 
\begin{df}
Let $(\R,\a,\b)$ be a BiHom-Lie conformal superalgebra. Then a
BiHom-conformal linear map $f_\lambda: \R\rightarrow \R$ is called an $\a^k\b^l$-derivation of $(\R,\a,\b)$ if
\begin{eqnarray}
&&f_\l \circ \a=\a\circ f_\lambda,  f_\lambda \circ \b=\b \circ f_\lambda, \nonumber\\
&&f_\l([a_\mu b])=[f_\l(a)_{\l+\mu}\a^{k}\b^l(b)]+(-1)^{|a||f|}[\a^{k}\b^l(a)_{\mu}f_\l(b)].
\end{eqnarray}
for all homogenous elements $a,b \in \R$.
\end{df}
Denote by $\textrm{Der}_{\a^{k}\b^l}$ the set of $\a^{k}\b^l$-derivations of the BiHom-Lie conformal
superalgebra $(\R,\a,\b)$. For any $a\in  \R$ satisfying $\a(a)=a, \b(a)=a$, define $f_{k,l}: \R\rightarrow \R$ by
\begin{eqnarray*}
f_{k,l}(a)_{\l}(b)=[a_\lambda\a^{k+1}\b^{l-1}(b)],~~~\forall b\in \R.
\end{eqnarray*}
Then $f_{k,l}(a)$ is an $\a^{k+1}\b^l$-derivation, which is called an inner $\a^{k+1}\b^l$-derivation. In fact,
\begin{eqnarray*}
f_{k,l}(a)_{\lambda}(\partial b)&=&[a_\lambda\a^{k+1}\b^{l-1}(\partial b)]
= [a_\lambda\partial\a^{k+1}\b^{l-1}( b)]
= (\partial+\lambda)f_{k,l}(a)_{\lambda}(b),
\end{eqnarray*}
\begin{eqnarray*}
&&f_{k,l}(a)_{\lambda}(\a(b))=[a_\lambda \a^{k+2}\b^{l-1}(b)]
= \a[a_\lambda\a^{k+1}\b^{l-1}( b)]
= \a\circ f_{k,l}(a)_{\lambda}(b),\\
&&f_{k,l}(a)_{\lambda}(\b(b))=[a_\lambda\a^{k+1}\b^{l}(b)]
= \b[a_\lambda\a^{k+1}\b^{l-1}( b)]
= \b\circ f_{k,l}(a)_{\lambda}(b),\\
&&f_{k,l}(a)_{\lambda}([b_{\mu}c])= [a_\lambda \a^{k+1}\b^{l-1}([b_{\mu}c])
=  [\a\b(a)_\lambda [\a^{k+1}\b^{l-1}(b)_{\mu}\a^{k+1}\b^{l-1}(c)]\\
&&~~~~~~~~~~~~~~~~~~~= [\b(a)_\lambda\a^{k+1}\b^{l-1}(b)]_{\lambda
+(-1)^{|a||b|}\mu}\a^{k+1}\b^{l}(c)]+[\a^{k+1}\b^{l}(b)_{\mu}[\a(a)_{\lambda}\a^{k+1}\b^{l-1}(c)]]\\
&&~~~~~~~~~~~~~~~~~~~=[a_\lambda\a^{k+1}\b^{l-1}(b)]_{\lambda+\mu}\a^{k+1}\b^{l}(c)]
+(-1)^{|a||b|}[\a^{k+1}\b^{l}(b)_{\mu}[a_{\lambda}\a^{k+1}\b^{l-1}(c)]]\\
&&~~~~~~~~~~~~~~~~~~~=[f_{k,l}(a)_{\lambda}(b)_{\lambda+\mu}\a^{k+1}\b^{l}(c)]
+(-1)^{|a||b|}[\a^{k+1}\b^{l}(b)_{\mu}(f_{k,l}(a)_{\lambda}(c))].
\end{eqnarray*}
Denote by $\textrm{Inn}_{\a^k\b^l} (\R)$ the set of inner $\a^k\b^l$-derivations. For $f_{\l}\in \textrm{Der}_{\a^k\b^l}(\R)$ and $g_{\mu-\lambda}\in\textrm{ Der}_{\a^s\b^t}(R)$, define their commutator $[f_\lambda g]_{\mu}$ by
\begin{eqnarray}\label{[f,g]}
[f_\l g]_{\mu}(a)=f_\l(g_{\mu-\lambda}a)-(-1)^{|f||g|}g_{\mu-\l}(f_\l a),~~~\forall a\in \R.
\end{eqnarray}

\begin{lem}
For any $f_{\l}\in \textrm{Der}_{\a^k\b^l}(\R)$ and $g_{\mu-\lambda}\in \textrm{Der}_{\a^s\b^t}(\R)$, we have
\begin{eqnarray*}
[f_\lambda g] \in \textrm{Der}_{\a^{k+s}\b^{l+t}}(\R)[\lambda].
\end{eqnarray*}
\end{lem}

\begin{proof}
For any $a,b\in \R$, we have
\begin{eqnarray*}
&&[f_\lambda g]_{\mu}(\partial a)\\
&=&f_{\lambda}(g_{\mu-\lambda}\partial a)-(-1)^{|f||g|}g_{\mu-\lambda}(f_\lambda \partial a)\\
&=& f_\lambda((\partial+\mu-\lambda)g_{\mu-\lambda}a)+(-1)^{|f||g|}g_{\mu-\lambda}((\mu+\lambda)f_\lambda a)\\
&=& (\partial+\mu)f_{\lambda}(g_{\mu-\lambda}a)-(-1)^{|f||g|}(\partial+\mu)g_{\mu-\lambda}(f_{\lambda}a)\\
&=& (\partial+\mu)[f_\lambda g]_{\mu}(a).
\end{eqnarray*}
Moreover, we have
\begin{eqnarray*}
&&[f_\lambda g]_{\mu}([a_{\g}b])
= f_\lambda(g_{\mu-\lambda}[a_{\g}b])-(-1)^{|f||g|}g_{\mu-\lambda}(f_\lambda[a_{\g}b])\\
&=&f_\lambda([g_{\mu-\lambda}(a)_{\mu-\lambda+\gamma}\a^s\b^{t}(b)]
+(-1)^{|a||g|}[\a^s\b^{t}(a)_{\g}g_{\mu-\lambda}(b)])\\
&& -(-1)^{|f||g|}g_{\mu-\lambda}([f_\lambda(a)_{\lambda+\gamma}\a^k\b^{l}(b)]
+(-1)^{|f||a|}[\a^k\b^{l}(a)_{\g}f_{\lambda}(b)])\\
&=& [f_\lambda(g_{\mu-\lambda}(a))_{\mu+\gamma}\a^{k+s}\b^{l+t}(b)]
+(-1)^{|f||g||a|}[\a^{k}\b^{l}(g_{\mu-\lambda}(a))_{\mu-\lambda+\g}f_{\lambda}(\a^{s}\b^{t}(b)))\\
&&+(-1)^{|a||g|}[f_\lambda(\a^s\b^t(a))_{\lambda+\gamma}\a^k\b^l(g_{\mu-\lambda}(b))]
+(-1)^{|f||g|[a|}[\a^{k+s}\b^{l+t}(a)_\g(f_\lambda(g_{\mu-\lambda}(b)))]\\
&&-(-1)^{|f||g|}[(g_{\mu-\lambda}f_\lambda(a))_{\mu+\gamma}\a^{k+s}\b^{l+t}(b)]
-[\a^{s}\b^{t}(f_\lambda(a))_{\lambda+\gamma}(g_{\mu-\lambda}(\a^k\b^l(b)))]\\
&&-(-1)^{|f||g|[a|}[(g_{\mu-\lambda}(\a^{k}\b^{l}(a)))_{\mu-\lambda+\gamma}\a^{s}\b^{t}(f_{\lambda}(b))]
-(-1)^{|f||g|[a|}[\a^{k+s}\b^{l+t}(a)_{\lambda}(g_{\mu-\lambda}(f_\lambda(b)))]\\
&=&[([f_\lambda g]_\mu a)_{\mu+\g}\a^{k+s}\b^{l+t}(b)]+(-1)^{|f||g||a|}[\a^{k+s}\b^{l+t}(a)_{\g}([f_\lambda g]_{\mu}b)].
\end{eqnarray*}
Therefore, $[f_\lambda g] \in Der_{\a^{k+s}\b^{l+t}}(\R)[\lambda]$.
\end{proof}

Define
\begin{eqnarray}
\textrm{Der}(\R)=\bigoplus_{k,l l\geq 0}\textrm{Der}_{\a^k\b^l}(\R).
\end{eqnarray}

\begin{prop}
$(\textrm{Der}(\R), \a',\b')$ is a BiHom-Lie conformal superalgebra with respect to \eqref{[f,g]} where $\a'(f)=f\circ \a$  and  $\b'(f)=f\circ \b$.
\end{prop}

\begin{proof}
It is straightforward. So we omit details.
\end{proof}

Let $\R$ be a BiHom-Lie conformal superalgebra. Define $\Der(\R)_\theta$ is the set of all derivations of degree $\theta$, then it is obvious that $\Der(\R)=\Der(\R)_0\oplus\Der(\R)_1$ is a subalgebra of $Cend(\R)$.

Define
\begin{align*}
\Omega=\{f_\lambda\in gc(\R)|f_\lambda \circ \a=\a\circ f_\lambda, f_\lambda \circ \b=\b\circ f_\lambda \}.
\end{align*}
Let $\a',\b': \Omega \to \Omega,\ \a'(f_\l)=f_\l \circ \a, \ \b'(f_\l)=f_\l \circ \b$.
Then $(\Omega,\a',\b')$ is a BiHom-Lie conformal algebra with respect to \eqref{crochet de cend}  and $\Der(\R)$ is a BiHom-conformal subalgebra of $\Omega$.

\begin{df}
An element $f$ in $\Omega$ is called
\begin{itemize}
\item [(a)] an $\a^k\b^l$-{\it generalized derivation}  of $\R$, if there exist $f^{'},f^{''}\in \Omega$ such that $|f|=|f'|=|f''|$ and
\begin{eqnarray}\label{generalized derivation}
[(f_\lambda(a))_{\lambda+\mu}\a^k\b^l(b)]+(-1)^{|f||a|}[\a^k\b^l(a)_\mu(f^{'}_\lambda(b))]
=f^{''}_\lambda([a_\mu b]), 
\end{eqnarray}
for all homogenous elements $a,b \in \R$.
\item  [(b)] an $\a^k\b^l$-{\it quasiderivation } of $\R$, if there is $f^{'}\in \Omega$ such that
$|f|=|f'|$ and
\begin{eqnarray}\label{quasiderivation}
[(f_\lambda(a))_{\lambda+\mu}\a^k\b^l(b)]+(-1)^{|f||a|}[\a^k\b^l(a)_\mu(f_\lambda(b))]
=f^{'}_\lambda([a_\mu b]), 
\end{eqnarray}
for all homogenous elements $a,b \in \R$.
\item   [(c)] an $\a^k\b^l$-{\it centroid} of $\R$, if it satisfies
\begin{eqnarray}\label{centroid}
 [(f_\lambda(a))_{\lambda+\mu}\a^k\b^l(b)]=(-1)^{|f||a|}[\a^k\b^l(a)_\mu(f_\lambda(b))]
 =f_\lambda([a_\mu b]), 
\end{eqnarray}
for all homogenous elements $a,b \in \R$.
\item   [(d)]  an $\a^k\b^l$-{\it quasicentroid}  of $\R$, if it satisfies
\begin{eqnarray}\label{quasicentroid}
[(f_\lambda(a))_{\lambda+\mu}\a^k\b^l(b)]=(-1)^{|f||a|}[\a^k\b^l(a)_\mu(f_\lambda(b))], 
\end{eqnarray}
for all homogenous elements $a,b \in \R$.
\item   [(e)] an $\a^k\b^l$-{\it central derivation}  of $\R$, if it satisfies
\begin{eqnarray}\label{centralderivation}
[(f_\lambda(a))_{\lambda+\mu}\a^k\b^l(b)]=f_\lambda([a_\mu b])=0, 
\end{eqnarray}
for all homogenous elements $a,b \in \R$.
\end{itemize}
\end{df}

Denote by $\textrm{GDer}_{\a^k\b^l}(\R)$, $\textrm{QDer}_{\a^k\b^l}(\R)$, $\textrm{C}_{\a^k\b^l}(\R)$, $\textrm{QC}_{\a^k\b^l}(\R)$ and $\textrm{ZDer}_{\a^k\b^l}(\R)$ the sets of all $\a^k\b^l$-generalized derivations,
 $\a^k\b^l$quasiderivations, $\a^k\b^l$centroids, $\a^k\b^l$quasicentroids and $\a^k\b^l$central derivations of $\R$. Set
 \begin{eqnarray*}
&&\textrm{GDer}(\R):=\bigoplus_{k\geq0, l\geq 0}\textrm{GDer}_{\a^k\b^l}(\R),~~\textrm{QDer}(\R):=\bigoplus_{k\geq0, l\geq 0}\textrm{QDer}_{\a^k\b^l}(\R).\\
&& \textrm{C}(\R):= \bigoplus_{k\geq0, l\geq 0} \textrm{C}_{\a^k\b^l}(\R),~~~\textrm{QC}_{\a^k\b^l}(\R):=  \bigoplus_{k\geq0, l\geq 0} \textrm{QC}_{\a^k\b^l}(\R),\\
&&\textrm{ZDer}(\R):=\bigoplus_{k\geq0, l\geq 0}\textrm{ZDer}_{\a^k\b^l}(\R).
\end{eqnarray*}

It is easy to see that
\begin{eqnarray}
\textrm{ZDer}(\R)\subseteq \textrm{Der}(\R)\subseteq \textrm{QDer}(\R)\subseteq \textrm{GDer}(\R)\subseteq Cend(\R),\,
\textrm{C}(\R)\subseteq \textrm{QC}(\R)\subseteq \textrm{GDer}(\R).
\end{eqnarray}

\begin{prop}
 Let $(\R,\a,\b)$ be a  BiHom-Lie conformal superalgebra. Then
\begin{itemize}
\item [(i)] $\textrm{GDer}(\R), \textrm{QDer}(\R)$  and $\textrm{C}(\R)$ are BiHom-Lie conformal subalgebras of $\Omega$,
\item [(ii)] $\textrm{ZDer}(\R)$ is a BiHom-Lie conformal  ideal of $\textrm{Der}(\R)$.
\end{itemize}
\end{prop}

\begin{proof}
(i)  We only prove that $GDer(\R)$ is a BiHom-conformal subalgebra of $\Omega$. The proof for the other two cases  can be done similarly.

For $f_\l \in GDer_{\a^k\b^l}(\R), g_\l \in GDer_{\a^s\b^t}(\R)$, $a,b\in \R$,
there exist $f^{'},f^{''}\in \Omega$  (resp. $g^{'},g^{''}\in \Omega$ ) such that Eq.\eqref{generalized derivation} holds for $f$ (resp. $g$).
We only need to show
\begin{eqnarray}\label{5-6}
[f^{''}_\lambda g^{''}]_\theta([a_\mu b])=[([f_\lambda g]_\theta(a))_{\mu+\theta} \a^{k+s}\b^{l+t}(b)]+(-1)^{(|f|+|g|)|a|}[\a^{k+s}\b^{l+t}(a)_\mu([f^{'}_\lambda g^{'}]_\theta(b))].
\end{eqnarray}
In fact, we have
\begin{align}
&[([f_\lambda g]_\theta(a))_{\mu+\theta}\a^{k+s}\b^{l+t}(b)]
=[(f_\lambda(g_{\theta-\lambda}(a)))_{\mu+\theta}\a^{k+s}\b^{l+t}(b)] \nonumber\\
& \hspace{3 cm} -(-1)^{|f||g|}[(g_{\theta-\lambda}(f_\lambda(a)))_{\mu+\theta}\a^{k+s}\b^{l+t}(b)]. \label{5-7}
\end{align}
On the other hand,
\begin{eqnarray}
&&[(f_\lambda(g_{\theta-\lambda}(a)))_{\mu+\theta}\a^{k+s}\b^{l+t}(b)]\nonumber\\
&=&f^{''}_\lambda([(g_{\theta-\lambda}(a))_{\mu+\theta-\lambda}\a^s\b^t(b)])-
(-1)^{|f|(|g|+|a|)}[\a^k\b^l(g_{\theta-\lambda}(a))_{\mu+\theta-\lambda}(f^{'}_\lambda(\a^s\b^t(b)))]\nonumber\\
&=&f^{''}_\lambda(g^{''}_{\theta-\lambda}([a_\mu b]))-(-1)^{|g||a|}f^{''}_\lambda([\a^s\b^t(a)_\mu (g^{'}_{\theta-\lambda}(b))])
\label{5-8} \\
&&-(-1)^{|f|(|g|+|a|)}g^{''}_{\theta-\lambda}([\a^s\b^t(a)_\mu(f^{'}_\lambda(b))])
+(-1)^{|f|(|g|+|a|)
+|g||a|}[\a^{k+s}\b^{l+t}(a)_\mu(g^{'}_{\theta-\lambda}(f^{'}_\lambda(b)))], \nonumber
\end{eqnarray}
and
\begin{eqnarray}
&&[(g_{\theta-\lambda}(f_\lambda(a)))_{\mu+\theta}\a^{k+s}\b^{l+t}(b)]\nonumber\\
&=&g^{''}_{\theta-\lambda}([(f_\lambda(a))_{\lambda+\mu}\a^k\b^l(b)])
-(-1)^{|g|(|f|+|a|)}[f_\lambda(\a^s\b^t(a))_{\lambda+\mu}(g^{'}_{\theta-\lambda}(\a^k\b^l(b)))]\nonumber\\
&=&g^{''}_{\theta-\lambda}(f^{''}_\lambda([a_\mu b]))-(-1)^{|f||a|}g^{''}_{\theta-\lambda}([\a^k\b^l(a)_\mu(f^{'}_\lambda(b))]) \label{5-9} \\
&&-(-1)^{|g|(|f|+|a|)}f^{''}_\lambda([\a^s\b^t(a)_\mu(g^{'}_{\theta-\lambda}(b))])
+(-1)^{|g|(|f|+|a|)
+|f||a|}[\a^{k+s}\b^{l+t}(a)_\mu f^{'}_\lambda(g^{'}_{\theta-\lambda}(b))]. \nonumber
\end{eqnarray}

Substituting Eqs.\eqref{5-8} and \eqref{5-9} into Eq.\eqref{5-7}, we obtain  Eq.\eqref{5-6}.
In addition, It is obvious that $\a'(f) \in \GDer(\R)$ and $\b'(f) \in \GDer(\R)$.
Hence $[f_\lambda g]\in\GDer(\R)[\lambda]$,
and $\GDer(\R)$ is a BiHom-conformal subalgebra of $\Omega$.

(ii) For $f\in\ZDer_{\a^k\b^l}(\R), g\in\Der_{\a^s\b^t}(\R)$, and $a, b \in\R$, we have
\begin{eqnarray*}
[f_\lambda g]_\theta([a_\mu b])
&=&f_\lambda(g_{\theta-\lambda}([a_\mu b]))-(-1)^{|f||g|}g_{\theta-\lambda}(f_\lambda([a_\mu b]))=f_\lambda(g_{\theta-\lambda}([a_\mu b]))\\
&=&f_\lambda([(g_{\theta-\lambda}(a))_{\mu+\theta-\lambda}\a^s\b^t(b)]
+(-1)^{|a||g|}[\a^s\b^t(a)_\mu (g_{\theta-\lambda}(b))])=0,
\end{eqnarray*}
and
\begin{eqnarray*}
{[[f_\lambda g]_\theta(a)_{\mu+\theta}\a^{k+s}\b^{l+t}(b)]}
&=&[(f_\lambda(g_{\theta-\lambda}(a))-
(-1)^{|f||g|}g_{\theta-\lambda}(f_\lambda(a)))_{\mu+\theta}\a^{k+s}\b^{l+t}(b)]\\
&=&[-(-1)^{|f||g|}(g_{\theta-\lambda}(f_\lambda a)))_{\mu+\theta}\a^{k+s}\b^{l+t}((b)]\\
&=&-(-1)^{|f||g|}g''_{\theta-\lambda}([f_{\lambda}(a)_{\lambda+\mu}\a^k\b^l(b)]) \\
&&+(-1)^{|g||a|}[f_\lambda (\a^s\b^t(a))_{\lambda+\mu}g'_{\theta-\lambda}(\a^k\b^t(b))]\\
&=&0.
\end{eqnarray*}
This means that
$[f_\lambda g] \in\ZDer(\R)[\lambda]$. Thus $\ZDer(\R)$ is an ideal of $\Der(\R)$.
\end{proof}

\begin{prop}
Let $f \in QC_{\a^k\b^l}(\R)$  and $g \in QC_{\a^s\b^l}(\R)$. Then $[f_\l g ]$ is an $\a^{k+s}\b^{l+t}$-generalized derivation of degree $|f|+|g|$.
\end{prop}

\begin{proof}
Assume that  $f \in QC_{\a^k\b^l}(\R)$  and $g \in QC_{\a^s\b^l}(\R)$. Then
for any homogenous elements  $a,b \in \R$, we have
\begin{align*}
[(f_\lambda(a))_{\lambda+\mu}\a^k\b^l(b)]=(-1)^{|f||a|}[\a^k\b^l(a)_\mu(f_\lambda(b))]
\end{align*}
and
\begin{align*}
[(g_\lambda(a))_{\lambda+\mu}\a^s\b^t(b)]=(-1)^{|g||a|}[\a^s\b^t(a)_\mu(g_\lambda(b))].
\end{align*}
Hence
\begin{align*}
[[f_\l g]_\theta (a)_{\mu+\theta} \a^{k+s}\b^{l+t}(b)] = &
[f_\l(g_{\theta-\l}(a))_{\mu+\theta}  \a^{k+s}\b^{l+t}(b)]  \\
&-(-1)^{|f||g|}[g_{\theta-\l}(f_\l(a))_{\mu+\theta}  \a^{k+s}\b^{l+t}(b)]  \\
=&(-1)^{(|f|+|g|)|a|}[ \a^{k+s}\b^{l+t}(a)_\mu g_{\theta-\l}(f_\l(b))] \\
 &-(-1)^{(|f|+|g|)|a|}(-1)^{|f||g|} [ \a^{k+s}\b^{l+t}(a)_\mu f_\l(g_{\theta-\l}(b))] \\
 =&-(-1)^{(|f|+|g|)|a|} [ \a^{k+s}\b^{l+t}(a)_\mu (-1)^{|f||g|}[f_\l g]_\theta (b)].
\end{align*}
Which implies that
\begin{align*}
 [[f_\l g]_\theta (a)_{\mu+\theta} \a^{k+s}\b^{l+t}(b)] + (-1)^{(|f|+|g|)|a|} [ \a^{k+s}\b^{l+t}(a)_\mu (-1)^{|f||g|}[f_\l g]_\theta (b)]=0.
\end{align*}
Then $[f_\l g ]$ is an $\a^{k+s}\b^{l+t}$-generalized derivation of degree $|f|+|g|$.
\end{proof}

\begin{prop}
Let $(\R,\a,\b)$ be a BiHom-Lie conformal superalgebra. If $f \in \QDer_{\a^k\b^l} (\R)$ and $g \in \QC_{\a^k\b^l}(\R)$ such that $|f|=|g|$, then $f+g \in \GDer_{\a^k\b^l}(\R)$ with degree $|f|$.
\end{prop}

\begin{proof}
Let  $f_\l \in\QDer_{\a^k\b^l}(\R)$. Then  there exist $f^{'} \in \Omega$ such that
\begin{eqnarray}\label{5-10}
[(f_\lambda(a))_{\lambda+\mu}\a^k\b^l(b)]
+(-1)^{|f||a|}[\a^k\b^l(a)_\mu(f_\lambda(b))]=f^{'}_\lambda([a_\mu b]), \forall \, a,b\in\R.
\end{eqnarray}
And let $g \in \QC_{\a^k\b^l}(\R)$. Then
\begin{eqnarray}\label{quasicentroid}
[(g_\lambda(a))_{\lambda+\mu}\a^k\b^l(b)]=(-1)^{|f||a|}[\a^k\b^l(a)_\mu(g_\lambda(b))], \ \forall \ a,b\in\R.
\end{eqnarray}
We will  prove  that $f+g$ satisfies  Eq. \eqref{generalized derivation}.  In fact,
\begin{align*}
 [(f+g)_\l(a)_{\l+\mu} \a^k\b^l(b)]
 =& [f_\l(a)_{\l+\mu} \a^k\b^l(b)]+ [g_\l(a)_{\l+\mu} \a^k\b^l(b)] \\
 =&f'_\l([a_\mu b]) -(-1)^{|f||a|}[\a^k\b^l(a)_\mu f_\l (b)] \\
 &+  (-1)^{|f||a|}[\a^k\b^l(a)_\mu g_\l (b)] \\
 =&f'_\l([a_\mu b]) - (-1)^{|f||a|}[\a^k\b^l(a)_\mu (f-g)_\l (b)].
\end{align*}
Then
\begin{align*}
 [(f+g)_\l(a)_{\l+\mu} \a^k\b^l(b)] +(-1)^{|f||a|}[\a^k\b^l(a)_\mu (f-g)_\l (b)]
 =f'_\l([a_\mu b]) .
\end{align*}
Therefore $f+g$ satisfies  Eq. \eqref{generalized derivation}. This completes the proof.
\end{proof}

\renewcommand{\refname}{References}

\end{document}